\documentclass[12pt]{article}
\usepackage[utf8]{inputenc}

\usepackage[usenames,dvipsnames,svgnames,table]{xcolor}

\renewcommand{\thefootnote}{\fnsymbol{footnote}}

\usepackage{epsfig}
\usepackage{cite}
\usepackage{latexsym,amsmath,amsthm}
\usepackage{amsfonts}
\usepackage{float}
\usepackage{pict2e}
\usepackage{tikz}
\usepackage{caption}
\usepackage{amsfonts}
\usepackage{amssymb}
\usepackage{mathrsfs}
\usepackage{amsmath}
\usepackage{enumerate}
\usepackage{graphicx}
\usepackage{subfig}
\usepackage{MnSymbol}
\usepackage{mathtools}
\usepackage{thm-restate}
\usepackage[unicode=true]{hyperref}
\hypersetup{
colorlinks,
linkcolor={blue!60!black},
citecolor={black},
urlcolor={blue!60!black}}
\usepackage{cleveref}

\renewcommand{\le}{\leqslant}
\renewcommand{\ge}{\geqslant}
\DeclareMathOperator{\ch}{ch}

	\newtheorem{theorem}{Theorem}
	\newtheorem{corollary}[theorem]{Corollary}
	\newtheorem{lemma}[theorem]{Lemma}
	\newtheorem{conjecture}{Conjecture}

	\theoremstyle{definition}
	\newtheorem{definition}{Definition}
	
	\newtheorem{claim}{Claim}

\begin{document}
\title{\Large \bf  Refined list  version of Hadwiger's Conjecture}
\author{%
Yangyan Gu\footnotemark[3] \and 
Yiting Jiang\footnotemark[3] \and 
David R.~Wood\footnotemark[2] \and 
Xuding Zhu\footnotemark[3]}
\date{}
\maketitle

\footnotetext[0]{\today}

\footnotetext[3]{College of Mathematics and Computer Sciences, Zhejiang Normal University, China (\texttt{\{ytjiang,
yangyan,xdzhu\}@zjnu.edu.cn}).}

\footnotetext[2]{School of Mathematics, Monash   University, Melbourne, Australia  (\texttt{david.wood@monash.edu}). Research supported by the Australian Research Council.}
		
\renewcommand{\thefootnote}{\arabic{footnote}}

\begin{abstract}
Assume $\lambda=\{k_1,k_2, \ldots, k_q\}$ is a partition of $k_{\lambda} = \sum_{i=1}^q k_i$. A  $\lambda$-list assignment of $G$ is a $k_\lambda$-list assignment $L$ of $G$ such that the colour set $\bigcup_{v \in V(G)}L(v)$ can be partitioned into $|\lambda|= q$ sets $C_1,C_2,\ldots,C_q$ such that for each $i$ and each vertex $v$ of $G$, $|L(v) \cap C_i| \ge k_i$. We say $G$ is \emph{$\lambda$-choosable} if $G$ is $L$-colourable for any $\lambda$-list assignment $L$ of $G$. The concept of $\lambda$-choosability is a refinement of   choosability that puts $k$-choosability and $k$-colourability in the same framework. If $|\lambda|$ is close  to $k_\lambda$, then $\lambda$-choosability is close  to $k_\lambda$-colourability; if $|\lambda|$ is close  to $1$, then $\lambda$-choosability is close  to $k_\lambda$-choosability. This paper studies Hadwiger's Conjecture in the context of $\lambda$-choosability. Hadwiger's Conjecture is equivalent to saying that every $K_t$-minor-free graph is $\{1 \star (t-1)\}$-choosable for any positive integer $t$. We prove that for   $t \ge 5$, for any partition $\lambda$ of $t-1$ other than $\{1 \star (t-1)\}$, there is a $K_t$-minor-free graph $G$ that is not $\lambda$-choosable. We then construct several types of $K_t$-minor-free graphs that are not $\lambda$-choosable,  where $k_\lambda - (t-1)$ gets larger as  $k_\lambda-|\lambda|$ gets larger.  In partcular,   for any $q$ and any $\epsilon > 0$, there exists $t_0$ such that for any $t \ge t_0$, for any partition $\lambda$ of $\lfloor (2-\epsilon)t \rfloor$ with $|\lambda| =q$,  there is a $K_t$-minor-free graph that is not $\lambda$-choosable. 
The $q=1$ case of this result was recently proved by Steiner, and our proof uses a similar argument.   We also generalize this result to $(a,b)$-list colouring. \bigskip
		
 \end{abstract}

\section{Introduction}

Given graphs $H$ and $G$, we say $H$ is {\em a minor} of $G$ (or $G$ has {\em an $H$-minor}) if a graph isomorphic to $H$ can be obtained from a subgraph of $G$ by contracting edges. Let $K_t$ be the $t$-vertex complete graph. A graph $G$ is {\em $K_t$-minor-free} if $G$ has no $K_t$-minor.  In 1943,  Hadwiger~\cite{Had} conjectured the following upper bound on the chromatic number of $K_t$-minor-free graphs:
		
\begin{conjecture}[Hadwiger's Conjecture]
\label{conj}
For every integer $t \ge 1$, every $K_t$-minor-free graph is $(t-1)$-colourable.
\end{conjecture}
		
This conjecture is a deep generalization of the Four Colour Theorem,   and has motivated many developments in graph colouring and graph minor theory. Hadwiger \cite{Had} and Dirac \cite{Dir} independently showed that Hadwiger's Conjecture  holds for $t \le  4$. Wagner \cite{Wan} proved that for $t=5$ the conjecture is equivalent to the Four Colour Theorem, which was subsequently proved by Appel, Haken and  Koch~\cite{AH,AHK} and Robertson, Sanders, Seymour and Thomas~\cite{RSST}, both using extensive computer assistance. Robertson, Seymour and Thomas \cite{RST} went one step further and proved Hadwiger's Conjecture for $t=6$, also by reducing it to the Four Colour Theorem. The conjecture for $t \ge 7$ is open and seems  to be extremely challenging. For more on Hadwiger's Conjecture, see the survey of Seymour~\cite{Sey}.
		
The evident difficulty of Hadwiger's Conjecture has inspired many researchers to study the following natural weakening (cf.\cite{RS1998,Kaw,KM}): 

\begin{conjecture}[Linear Hadwiger's Conjecture]
\label{conj1}
There exists a constant $C>0$ such that for every integer $t\ge1$,  every $K_t$-minor-free graph is $Ct$-colourable.
\end{conjecture}
		
For many decades, the best general upper bound on the chromatic number of   $K_t$-minor-free graphs was  $O(t\sqrt{\log t})$, which was proved independently   by Kostochka~\cite{Kos82,Kos84} and Thomason~\cite{Tho} in the 1980s. In 2019,  Norine, Postle and Song \cite{NSP}  broke this barrier, and proved  that the maximum chromatic number of $K_t$-minor-free graphs is in $O(t(\log t) ^{1/4+o(1)})$. Following a series of improvements~\cite{NP,Pos1,Pos2,Pos3}, the best known bound is $O(t\log \log t)$ due to Delcourt and Postle~\cite{DP}.
		
A list assignment of a graph $G$ is a mapping $L$ that assigns to each vertex $v$ of $G$ a set $L(v)$ of permissible colours. An {\em $L$-colouring} of $G$ is a proper colouring $f$ of $G$ such that for each vertex $v$ of $G$, $f(v)\in L(v)$. We say $G$ is {\em $L$-colourable} if $G$ has an $L$-colouring. A {\em $k $-list assignment} of $G$ is a list assignment $L$ with $|L(v)| \ge k$ for each vertex $v$. We say $G$ is {\em $k$-choosable} if $G$ is $L$-colourable for any $k$-list assignment $L$ of $G$. The {\em choice-number} of $G$ is the minimum integer $k$ such that $G$ is $k$-choosable.
		    
Hadwiger's Conjecture is also widely considered in the setting of list colourings. Voigt~\cite{Voi} constructed planar graphs (hence $K_5$-minor-free) with choice-number $5$. Hence the list version of Hadwiger's Conjecture is false. Nevertheless,  the list version of Linear Hadwiger's Conjecture,    proposed by Kawarabayashi and Mohar \cite{KM} in 2007, remains open.
		
\begin{conjecture}[List Hadwiger's Conjecture]
\label{conj2}
There exists a constant $C>0$ such that for every integer $t\ge1$,  every $K_t$-minor-free graph is $Ct$-choosable.
\end{conjecture}
		
The current state-of-the-art upper bound on the choice-number of $K_t$-minor-free graphs is $O(t(\log \log t)^2)$  \cite{DP}.  
  
If Conjecture~\ref{conj2} is true, then a natural problem is to determine the minimum value of $C$. Bar\'at, Joret and Wood~\cite{BJW} constructed $K_t$-minor-free graphs that are not $4(t-3)/3$-choosable, implying $C \ge \frac43$ in Conjecture~\ref{conj2}. Improving upon this result, Steiner~\cite{Ste} recently proved that the maximum choice-number of $K_t$-minor-free graphs is at least $2t-o(t)$, and hence $C \ge 2$ in Conjecture~\ref{conj2}.
  
\subsection{$\lambda$-Choosability}

In general, $k$-colourability and $k$-choosability behave very differently. Indeed, bipartite graphs can have arbitrary large choice-number. Zhu~\cite{Zhu} introduced a refinement of the concept of choosability, $\lambda$-choosability, that puts $k$-choosability and $k$-colourability in the same framework and considers a more complex hierarchy of colouring parameters.
		
\begin{definition}
Let $\lambda=\{k_1,k_2, \ldots, k_q\}$ be a multiset of positive integers. Let $k_{\lambda} = \sum_{i=1}^q k_i$ and $|\lambda| = q$. A {\em $\lambda$-list assignment} of $G$ is a list assignment $L$ such that the colour set $\bigcup_{v \in V(G)}L(v)$ can be partitioned into $q$ sets $C_1,C_2,\ldots,C_q$ such that for each $i$ and each vertex $v$ of $G$, $|L(v) \cap C_i| \ge k_i$.  We say $G$ is \emph{$\lambda$-choosable} if $G$ is $L$-colourable for any $\lambda$-list assignment $L$ of $G$.
\end{definition}
Note that for each vertex $v$, $|L(v)| \ge \sum_{i=1}^qk_i = k_{\lambda}$. So a $\lambda$-list assignment $L$ is a $k_{\lambda}$-list assignment with some restrictions on the set of possible lists.
		
		For a positive integer $a$, let $m_{\lambda}(a)$ be the multiplicity of $a$ in $\lambda$. If $m_{\lambda}(a)=m$, then instead of writing $m$ times the integer $a$, we write $a \star m$. For example, $\lambda=\{1 \star k_1, 2\star k_2,  3\}$ means that $\lambda$ is the multiset consisting of $k_1$ copies of $1$, $k_2$ copies of $2$ and one copy of $3$.
		If   $\lambda=\{k  \}$, then $\lambda$-choosability is the same as $k $-choosability; if  
		$\lambda = \{1 \star k\}$, then
		$\lambda$-choosability is equivalent to $k $-colourability. 
		So the concept of $\lambda$-choosability puts $k$-choosability and $k$-colourability in the same framework.

		For $\lambda = \{k_1,k_2,\ldots, k_q\}$ and $\lambda'= \{k'_1,k'_2,\ldots,k'_p\}$, we say $ \lambda'$ is a {\em refinement} of $\lambda$ if $p \ge q$ and there is a partition $I_1, I_2 , \ldots , I_q$ of $\{1,2,\ldots, p\}$ such that $\sum_{j \in I_t}k'_j = k_t$ for $t=1,2,\ldots, q$. We say $\lambda'$ is obtained from $\lambda$ by \emph{increasing} some parts if $p=q$ and $k_t \le k'_t$ for $t=1,2,\ldots, q$. We write $\lambda \le \lambda'$ if 
		$\lambda'$ is obtained from a refinement of $\lambda$ by increasing some parts. 
		It follows from the definitions that if $\lambda \le \lambda'$, then every $\lambda$-choosable graph is $\lambda'$-choosable. Conversely,  Zhu~\cite{Zhu} proved that 
		if $\lambda \not\le \lambda'$, then there is a $\lambda$-choosable graph that is not $\lambda'$-choosable.  
		In particular,  $\lambda$-choosability implies $k_{\lambda}$-colourability, and if $\lambda \ne \{1 \star k_{\lambda}\}$, then
		there are  $k_{\lambda}$-colourable graphs that are not $\lambda$-choosable.
		
		All the partitions $\lambda$ of a  positive integer $k$ are sandwiched between $\{k\}$ and $\{1 \star k\}$ in the above order. As observed above, $\{k\}$-choosability is the same as $k $-choosability, and $\{1 \star k\}$-choosability is equivalent to $k $-colourability. By considering other partitions $\lambda$ of $k$, $\lambda$-choosability provides a complex hierarchy of colouring parameters that interpolate between $k$-colourability and $k$-choosability. 
		
		The framework of $\lambda$-choosability provides room to explore strengthenings of colourability and choosability results.  For example, Kermnitz and Voigt~\cite{VoigtNOT112} proved that there are planar graphs that are not $\{1,1,2\}$-choosable. This result strengthens  Voigt's result that there are non-4-choosable planar graphs, and shows that the Four Colour Theorem is sharp in the sense that for any partition $\lambda$ of $4$ other than $\{1 \star 4\}$, there is a planar graph that is not $\lambda$-choosable.
  
		This paper considers Hadwiger's Conjecture in the context of $\lambda$-choosability.  Conjectures~\ref{conj}, \ref{conj1} and \ref{conj2} can be restated in the language of $\lambda$-choosability as follows:

		\medskip\noindent
		\textbf{Conjecture~\ref{conj}'.} \textit{For every integer $t \ge 1$,  every $K_t$-minor-free graph is $\{1\star (t-1)\}$-choosable.}
		
		\medskip\noindent
		\textbf{Conjecture~\ref{conj1}'.} \textit{There exists a constant $C>0$ such that for every integer $t\ge1$,  every $K_t$-minor-free graph is $\{1\star Ct\}$-choosable.}
		
		\medskip\noindent
		\textbf{Conjecture~\ref{conj2}'.} \textit{There exists a constant $C>0$ such that for every integer $t\ge 1$,  every $K_t$-minor-free graph is $\{Ct\}$-choosable.}
		

\subsection{Results}
  
	This paper constructs several examples of $K_t$-minor-free graphs that are not $\lambda$-choosable where $k_\lambda\ge t-1$ and $q$ is close to $k_\lambda$. In particular, if   the multiplcity of 1 in $\lambda$ is large enough, then the number of parts of $\lambda$ will be close to $k_\lambda$.
		
  First we strengthen the above-mentioned  result  of Kermnitz and Voigt  to $K_t$-minor-free graphs for $t\ge 5$ as follows:
  
		\begin{restatable}{theorem}{thmI}
		\label{thm1}
		For every integer $t\ge 5$, there exists a $K_t$-minor-free graph that is not $\{1\star (t-3),2\}$-choosable.
		\end{restatable}

If $\lambda$ is a partition of  $t-1$ other than $\{1 \star (t-1)\}$, then $\{1\star (t-3),2\}$ is a refinement of $\lambda$. Hence we have the following corollary.

 \begin{corollary}
 \label{cor-tight1}
     If $\lambda$ is a partition of  $t-1$ other than $\{1 \star (t-1)\}$, then there is a $K_t$-minor-free graph that is not $\lambda$-choosable.
 \end{corollary}

For a multiset $\lambda$ of positive integers, let $h(\lambda)$ be the maximum $t$ such that every $K_t$-minor-free graph  is $\lambda$-choosable. Since $K_{k_\lambda+1}$ is not $k_\lambda$-colourable and  hence not $\lambda$-choosable, we know that $h(\lambda) \le k_{\lambda}+1$. 

For a multiset $\lambda$ of positive integers, $k_\lambda - |\lambda|$ measures the ``distance" of $\lambda$-choosability from $k_{\lambda}$-colourability. 
Hadwiger's Conjecture says that if $k_{\lambda} - |\lambda| = 0$, then $h(\lambda) =k_{\lambda}+1$. By Theorem~\ref{thm1}, if $k_{\lambda}- |\lambda| \ge 1$, then $h(\lambda) \le k_\lambda$, provided that $k_{\lambda} \ge 5$. It seems natural that if $k_\lambda - |\lambda|$ gets bigger, then $k_\lambda - h(\lambda)$ also gets bigger, provided that $k_\lambda$ is sufficiently large. The next result shows this is true for various $\lambda$.
		
\begin{restatable}{theorem}{thmII}
\label{thm2}
For each integer $a\ge 0$, there exists an integer $t_1=t_1(a)$ such that for every integer $t\ge t_1$, there exists a $K_t$-minor-free graph that is not $\{1\star (t-2a-6) ,3a+6\}$-choosable.
\end{restatable}

For the $\lambda$ in Theorem~\ref{thm2}, $k_\lambda=t+a$, $h(\lambda) \le t-1$ and $|\lambda|=t-(2a+5)$. As $k_\lambda - |\lambda| = 3a+5$ tends to infinity, the difference $k_\lambda - h(\lambda) \ge a+1$ also tends to infinity, provided that $k_\lambda \ge \phi(k_\lambda - |\lambda|)$, where $\phi$ is a certain given function.
It remains open whether such a conclusion holds for all $\lambda$. We conjecture a positive answer.

\begin{conjecture}
\label{conj4}
There are functions $\phi, \psi: \mathbb{N} \to \mathbb{N}$  for which the following hold:
\begin{itemize}
    \item $\lim_{n \to \infty} \psi(n) = \infty$. 
    \item For any multiset $\lambda$ of positive integers, if $k_\lambda \ge \phi(k_\lambda - |\lambda|)$,  then $k_\lambda - h(\lambda) \ge \psi(k_\lambda - |\lambda|)$.
\end{itemize}
\end{conjecture}

It is easy to see that if  $k_\lambda - |\lambda| = b$, then $\{1 \star (k_\lambda-2b'), 2\star b' \}$ is a refinement of $\lambda$, where $b \ge b' \ge b/2$. Thus to prove Conjecture~\ref{conj4}, it suffices to prove it for $\lambda$ of the form $\{1 \star k_1, 2 \star k_2\}$.

Theorem~\ref{thm3} below shows that Conjecture~\ref{conj4} holds for any $\lambda$ of the form $\{1 \star k_1, 3 \star k_2\}$.

		\begin{restatable}{theorem}{thmIII}
		\label{thm3}
			For each integer $a\ge 0$, there exists an integer $t_2=t_2(a)$ such that for every integer $t\ge t_2$, there exists a $K_t$-minor-free graph that is not $\{1\star (t-5a-9),3\star (2a+3)\}$-choosable.
		\end{restatable}


   As $|\lambda|$ becomes very small compared to $k_\lambda$, say $|\lambda|$ is constant and $k_\lambda$ tends to infinity, then    $\lambda$-choosability becomes very close to $k_\lambda$-choosability. The following result, which  generalizes the main result of Steiner~\cite{Ste}, deals with such $\lambda$. 
		
		\begin{restatable}{theorem}{thmkq}
		\label{thm:kq}
			For every $\varepsilon\in (0,1)$ and $q\in \mathbb{N}$, there exists an integer $t_3=t_3(q,\varepsilon)$ such that for every integer $t\ge t_3$ and $k_1,k_2,\dots,k_q\in \mathbb{N}$ satisfying
			$$  \sum_{j=1}^{q}k_j\le (2-\varepsilon)t,$$
			there exists a $K_t$-minor-free graph $G$ that is not $\{k_1,k_2,\dots,k_q\}$-choosable.
		\end{restatable}

  The $q=1$ case of Theorem~\ref{thm:kq} was proved by Steiner~\cite{Ste}. 
		
\subsection{Fractional Colouring}
  
	Next we consider the fractional version of Hadwiger's Conjecture. A $b$-fold colouring of a graph $G$ is a mapping $\phi$ that assigns to each vertex $v$ of $G$ a set $\phi(v)$ of $b$ colours, so that adjacent vertices receive disjoint colour sets. An $(a, b)$-colouring of $G$ is a $b$-fold colouring $\phi$ of $G$ such that $\phi(v) \subseteq \{1, 2,\dots, a\}$ for each vertex $v\in V(G)$. The \emph{fractional chromatic number} of $G$ is
	$$\chi_f(G) := \inf \left\{\frac ab:G \text{ is }(a,b)\text{-colourable} \right\}.$$

  The fractional version of Hadwiger's Conjecture was studied by Reed and Seymour~\cite{RS1998}, who proved that every $K_t$-minor-free graph $G$ has fractional chromatic number at most $2t$.
  
	An $a$-list assignment of $G$ is a mapping $L$ that assigns to each vertex $v$ a set $L(v)$ of $a$ permissible colours. A $b$-fold $L$-colouring of $G$ is a $b$-fold colouring $\phi$ of $G$ such that $\phi(v) \subseteq L(v)$ for each vertex $v$. We say $G$ is \emph{$(a, b)$-choosable} if for any $a$-list assignment $L$ of $G$, there is a $b$-fold $L$-colouring of $G$. The \emph{fractional choice-number} of $G$ is
	$$\ch_f(G) := \inf \left\{\frac ab:G \text{ is }(a,b)\text{-choosable} \right\}.$$
		
	Alon, Tuza and Voigt~\cite{ATV} proved that for any graph $G$, if $G$ is $(a,b)$-colourable, then $G$ is $(am, bm)$-choosable for some integer $m$. So for any graph $G$, $\chi_f(G) = \ch_f(G)$, and moreover the infimum in the definition of $\ch_f(G)$ is attained and hence can be replaced by minimum. 

We prove the following result by an argument parallel to the proofs  in   \cite{Ste}.
		
\begin{restatable}{theorem}{thmmult}
\label{thm-1}
Let $\varepsilon\in (0,1)$ be fixed. For every positive integer $m$, there exists $t_0=t_0(\varepsilon)$ such that for every integer $t\ge t_0$ there exists a $K_t$-minor-free graph $G$ that is not $((2-\varepsilon)tm,m)$-choosable.
\end{restatable}
		
Note that the graph $G$ in Theorem~\ref{thm-1} depends on  $m$ (as well as on $\epsilon$).  The  result of Reed and Seymour implies that for every $K_t$-minor-free graph $G$, there is a constant $m$ such that $G$ is $(2tm,m)$-choosable. Here the integer $m$ depends on $G$. Theorem~\ref{thm-1} has no implication for the fractional choice number of $G$.

\section{A key lemma}	

Let $G_1$ and $G_2$ be graphs, and   $S_i$ be a $k$-clique in   $G_i$ for $i=1,2$.
We say a graph $G$ is a {\em $k$-clique-sum} of $G_1$ and $G_2$ (on $S_1$ and $S_2$), if $G$ is obtained from the disjoint union of $G_1$ and $G_2$ by identifying pairs of the vertices of $S_1$ and $S_2$ to form a single shared clique and then possibly deleting some of the clique edges. The lemma is well-known and easily proved. 
 
\begin{lemma}
\label{lem-3}
Let $G_1$ and $G_2$ be $K_t$-minor-free graphs. If $G$ is a $k$-clique-sum of   $G_1$ and $G_2$ (on cliques $S_1$ of $G_1$ and $S_2$ of $G_2$), then   $G$ is $K_t$-minor-free.
\end{lemma}

\begin{definition}
Assume $\lambda=\{k_1,k_2, \ldots, k_q\}$ is a multiset of positive integers, $G$ is a graph and  $K=\{v_1,v_2, \ldots, v_p\}$ is a  clique in $G$, and $\mathcal{C}=(C_1,C_2, \ldots, C_q)$ is a $q$-tuple of   disjoint colour sets. A  $(\lambda, \mathcal{C})$-list assignment of $G$ is a  list assignment $L$ of $G$ such that for each vertex $v$, $|L(v) \cap C_i| \ge k_i$. 
     
Assume $K'=\{v'_1, v'_2, \ldots, v'_p\}$ is a   $p$-clique (disjoint from $G$), and $\psi'$ is a proper colouring of $K'$.  A  $(\lambda, \mathcal{C})$-list assignment $L$ of $G$ is a {\em $\psi'$-obstacle } for $(G,K, \mathcal{C})$  if the colouring $\psi$ of $K$ defined as $\psi(v_i) = \psi'(v'_i)$ is an $L$-colouring of $K$ that cannot be  extended to a proper $L$-colouring of $G$. 
\end{definition}
  
The following lemma will be used in some of our proofs.
  
\begin{lemma}
	\label{lem-key}
	Let $t$ be a positive integer and  $\lambda=\{k_1,k_2, \ldots, k_q\}$ be a multiset of positive integers.
	Assume there are $K_t$-minor-free   graphs $H_1$ and $H_2$,   a  clique $K=\{v_1,v_2, \ldots, v_p\}$   in $H_1$, a $q$-tuple of   disjoint colour sets $\mathcal{C}=(C_1,C_2, \ldots, C_q)$ 
  and 
  a $(\lambda, \mathcal{C})$-list assignment $L$ of $H_2$,  for which the following holds:
\begin{itemize}
	\item  for any $L$-colouring $\psi$ of $H_2$, there is a   $p$-clique $K_{\psi} = \{v_{\psi,1}, v_{\psi,2}, \ldots, v_{\psi, p}\}$ in $H_2$, such that there 
exists a $\psi|_{K_{\psi}}$-obstacle $L_{\psi}$ for $(H_1,K,\mathcal{C})$.
\end{itemize}
Then there is a $K_t$-minor-free graph $G$ that is not $\lambda$-choosable.
\end{lemma}
\begin{proof}
	We shall construct a graph $G$ and a $\lambda$-list assignment $L'$ of $G$ so that $G$ is $K_t$-minor-free and $G$ is not $L'$-colourable.
	Let $H_1,H_2$ be  graphs, $K$ be a  $p$-clique in $H_1$, 
	and $L$ be a $(\lambda, \mathcal{C})$-list assignment of $H_2$, satisfying the assumption of the lemma.
	
	Now we start the construction of $G$ and $L'$. First we take  a copy of $H_2$, and let $L$ be the $(\lambda, \mathcal{C})$-list assignment of $H_2$
	as above.  
	
	For each  proper $L$-colouring $\psi$ of $H_2$, choose a $p$-clique   $K_{\psi}=\{v_{\psi,1}, v_{\psi,2}, \ldots, v_{\psi,p}\}$      in $H_2$, for which there is  a  $\psi|_{K_{\psi}}$-obstacle $L_{\psi}$  for $(H_1,K,\mathcal{C})$.  Take a copy $H_{\psi}$ of $H_1$.
	Let $K'_{\psi} =\{v'_{\psi,1}, v'_{\psi,2}, \ldots, v'_{\psi,p}\}$ be the copy of $K$ in $H_{\psi}$  (where $v'_{\psi,i}$ is the copy of $v_i \in V(K)$ in $K'_{\psi}$). Identify  $K_{\psi}$ with $K'_{\psi}$ in such a way that $v_{\psi,i}$ is identified with $v'_{\psi, i}$. Extend the list assignment $L$ to $V(H_{\psi}) - K'_{\psi}$ by letting $L'(v_{\psi})=L_{\psi}(v)$, where $v_{\psi}$ is the copy of $v \in H_1$ in $H_{\psi}$. 
	
	This completes the construction of the graph $G$ and the list assignment $L'$ of $G$. Observe that for each  proper $L$-colouring $\psi$ of $H_2$, we have chosen a $p$-clique $K_{\psi}=\{v_{\psi,1}, v_{\psi,2}, \ldots, v_{\psi,p}\}$ in $H_2$. A vertex $v$ of $H_2$ may be contained in many copies of $p$-cliques, say $v$ is contained in $K_{\psi_1}, K_{\psi_2}, \ldots, K_{\psi_s}$. Then $v$ has different names in these copies of cliques. The name for $v$ in  $K_{\psi}$ is only used to find its partner vertex (the vertex to be identified with $v$) in $H_{\psi}$. So this  leads to no confusion. 
	
	It follows from Lemma~\ref{lem-3} that $G$ is $K_t$-minor-free, and $L'$ is a $(\lambda, \mathcal{C})$-list assignment of $G$.

	Now we show that $G$ is not $L'$-colourable. Assume to the contrary that 
	there is a   proper $L'$-colouring $\phi$ of $G$. Let $\psi$ be the restriction of $\phi$ to $H_2$. The restriction of $\phi$ to $H_{\psi}$ is an $L_{\psi}$-colouring of     $H_{\psi}$. But $L_{\psi}$ is a $\psi|_{K_{\psi}}$-obstacle for $(H_1,K, \mathcal{C})$,   a contradiction. 
\end{proof}
		 
\section{Proofs of the theorems}

\thmI*

\begin{proof}
The proof is by induction on $t$. For $t=5$, a non-$\{1,1,2\}$-choosable planar graph (hence a $K_5$-minor-free graph) was constructed in~\cite{VoigtNOT112}. Assume $t \ge 6$ and there exists a $K_{t-1}$-minor-free graph $G_{t-1}$ that is not $\lambda$-choosable, where $\lambda=\{1\star(t-4),2\}$. 

Let $L$ be a $\lambda$-list assignment of $G_{t-1}$, such that $G_{t-1}$ is not $L$-colourable. Let $\mathcal{C}=(C_1,C_2, \ldots, C_{t-3})$ be a $(t-3)$-tuple of disjoint colour sets so that for each vertex $v$ of $G_{t-1}$, $|L(v) \cap C_i| =1$ for $1 \le i \le t-4$ and $|L(v) \cap C_{t-3}| =2$. Zhu~\cite{Zhu} showed that we may  assume     $C_i=\{c_i\}$  for $i=1,2,\ldots, t-4$.

 Let $H_1$ be the graph obtained from $G_{t-1}$ by adding a vertex $u$ adjacent to every vertex of $G_{t-1}$. Let 
 $$\mathcal{C}'=(C'_1,C'_2, \ldots, C'_{t-2})$$ 
 where $C'_i=C_i$ for $i=1,2,\ldots, t-4$, $C'_{t-3}=\{c_{i-3}\}$ and $C'_{t-2}=C_{t-3}$. Let $a,b$ be two colours from $C_{t-3}$. Let $\lambda'=\{1 \star (t-3),2\}$ and let $L'$ be the $(\lambda', \mathcal{C}')$-list assignment of $H_1$ defined as 
 \[
L'(v) = \begin{cases} L(v) \cup \{c_{t-3}\}, &\text{ if $v \in V(G_{t-1})$}, \cr 
\{c_1,c_2, \ldots, c_{t-3}, a,b\}, &\text{ if $v=u$}.
\end{cases}
 \]
 Let $K=\{u\}$ be the $1$-clique in $H_1$. 
 If $\psi$ is an $L'$-colouring of a copy of $K_1=\{u'\}$ with $\psi(u')=c_i$ for some $1 \le i \le t-3$, then $L'$ is a $\psi$-obstacle for $(H_1,K,\mathcal{C}')$.

 Let $H_2$ be a  triangle   and  $L''$ be the $(\lambda', \mathcal{C}')$-list assignment of $H_2$, defined as $L''(v)=\{c_1,c_2, \ldots,c_{t-3}, a,b\}$ for each vertex $v$ of $H_2$. Then for any proper $L''$-colouring $\psi$ of $H_2$, there is a vertex $v$ (a copy of $K_1$) such that $\psi(v) = c_i$ for some $1 \le i \le t-3$. Hence $L'$ is a $\psi|_{\{v\}}$-obstacle for $(H_1,K, \mathcal{C}')$.

By Lemma~\ref{lem-key}, there is a $K_t$-minor-free graph $G_t$ that is not $\lambda'$-choosable.
 \end{proof}
		
\thmII*
	
\begin{proof}	 
Assume $a$ is a positive integer. Let 
$$m := \binom{2a+5}{a+3} \quad \text{and}\quad t_1:=(2a+5)m+2.$$  
Assume   $t\ge t_1$. We shall construct a $K_t$-minor-free graph $G$ that is not $\{1 \star (t-2a-6), 3a+6\}$-choosable by using Lemma~\ref{lem-key}.

First, let $H_1$   be a graph with vertex set $A \cup B$  such that   $A \cap B = \emptyset$ and
	\begin{itemize}
			\item   $A$ induces a $(2a+5)$-clique,  $B$ induces a $(t-2)$-clique, 
			\item each vertex in $B$ has exactly $a+3$ neighbours in $A$, and
   \item for each $(a+3)$-subset $X$ of $A$, if $B_X := \{v \in B: N_{H_1}(v) \cap A = X\}$, then 
   $$| B_X| \ge \left\lfloor \frac {t-2}{m} \right\rfloor.$$
		\end{itemize}
		 It is easy to see that such a graph $H_1$ exists. 
		
		\begin{claim}
		\label{clm1}
			The graph $H_1$ is $K_t$-minor-free.
		\end{claim}
		\begin{proof}
			Assume that $H_1$ has a $K_t$-minor. Then   there exists a collection $\mathcal{Z}$ of $t$ non-empty and pairwise disjoint subsets of $V(H_1)$ such that   for each $Z \in \mathcal{Z}$, $H_1[Z]$ is connected, and for any two distinct $Z,Z' \in \mathcal{Z}$, there exists at least one edge in $H_1$ joining a vertex in $Z$ to a vertex in $Z'$. In particular, for any $Z \in \mathcal{Z}$, there are at least $(t-1)$ vertices in 
   $V(H_1)-Z$ adjacent to vertices in $Z$.
			
		  Since $|B|=t-2$,  there are at least two subsets $Z \in \mathcal{Z}$ that are contained in $A$. As $|A|=2a+5$, there  exists $Z \in \mathcal{Z}$ such that $Z \subseteq A$ and $|Z| \le a+2$. 

  Let $X$ be an $(a+3)$-subset of $A-Z$. Then 
			\begin{align*}
			|N_{H_1}(Z)| \le   |V(H_1)|-|B_X| 
			\le\;	&(2a+5)+(t-2)- \left\lfloor \frac {t-2}{m} \right\rfloor \\
				 <\; & t+2a+4-\frac{t-2}{m}\\
				 =\; & t+2a+4-\frac{t-2}{\binom{2a+5}{a+3}}\\
				 \le\; &  t+2a+4-(2a+5)\\
				 =\; & t-1, 
			\end{align*}
		  a contradiction. 
		\end{proof}

	Label the vertices in $A$ as $v_1,v_2,\dots,v_{2a+5}$.
	Let $$\{a_i: i \in [3a+6]\}, \ \{b_i:i\in [t-2a-6]\}, \ \{c_i:i\in [2a+3]\}$$   be pairwise disjoint colour sets.
	Let $\psi: A \to \{a_i: i \in [3a+6]\}$ be a injective mapping. 
	Let $L_\psi$ be the list assignment of $H_1$ defined as follows:
   \begin{itemize}
			    \item[(LA)] $L_\psi(v) = \{b_i:i\in[t-2a-6]\}\cup\{a_i:i\in[3a+6]\}$  for   $v\in A$.
			    \item[(LB)] $L_\psi(v)=\psi(N_A(v)) \cup \{b_i:i\in [t-2a-6]\}\cup\{c_i:i\in [2a+3]\}$ for $v \in B$.
			\end{itemize}
Let
	$$\mathcal{C}=(C_1,C_2, \ldots, C_{t-2a-5})$$ 
	where $C_i = \{b_i\}, \text{  for $i=1,2,\ldots, t-2a-6$},  \   C_{t-2a-5}=\{a_i: i \in [3a+6]\} \cup \{c_i:i\in [2a+3]\}.$
Let $\lambda = \{1\star (t-2a-6), 3a+6\}$.	Then $L_\psi$   is a $(\lambda, \mathcal{C})$-list assignment of $H_1$: 
  For each vertex $v$ of $H_1$, $|L_\psi(v) \cap C_i|=1$ for $i=1,2,\ldots, t-2a-6$, and $|L_\psi(v) \cap C_{t-2a-5}| = 3a+6$.
   Moreover, $\psi$ is an $L_\psi$-colouring of $A$. 
   
  \begin{claim}
	\label{clm2}		 
  $\psi$ cannot be extended to an $L_{\psi}$-colouring of $H_1$.
		\end{claim}
		\begin{proof}
			Assume that $H_1$ has an $L$-colouring $\phi_\psi $ which is an extension of $\psi$.
			Then $\phi(v)=\psi(v)$ for each $v\in A$ and $\phi(v)\in\{b_i:i\in [t-2a-6]\}\cup\{c_i:i\in [2a+3]\}$ for every vertex $v\in B$. 
   Thus the vertices of the  $(t-2)$-clique induced by $B$  are coloured by  $(t-2a-6)+(2a+3)=t-3$ colours, a contradiction.    
		\end{proof}

Let $H_2$ be a $(t-1)$-clique and $L'$ be the $\{1\star (t-2a-6),3a+6\}$-list assignment defined as $$L'(v)=\{b_i:i\in[t-2a-6]\}\cup\{a_i:i\in[3a+6]\},$$
for each vertex $v$ of $H_2$. Then for any proper $L'$-colouring $\psi$ of $H_2$, there is a $(2a+5)$-clique $K_\psi=\{v_{\psi,1},v_{\psi,2},\dots,v_{\psi,2a+5}\}$ in $H_2$  such that $\psi(v_{\psi,i})\in\{a_j: j \in [3a+6]\}$ for $i\in[2a+5]$.
 
By Claim~\ref{clm2}, $L_{\psi}$ is a $\psi|_{K_\psi}$-obstacle for $(H_1,H_1[A],\mathcal{C})$.

By Lemma~\ref{lem-key}, there is a $K_t$-minor-free graph $G$ that is not $\{1\star (t-2a-6),3a+6\}$-choosable.
  \end{proof}

\thmIII*

\begin{proof}
Assume $a$ is a positive integer. Let 
$$m :=3^{a+2} \quad\text{and}\quad t_2:=(2a+5)m+a+3.$$  
Assume   $t\ge t_2$. We shall construct a $K_t$-minor-free graph $G$ that is not  $\{1 \star (t-4a-9), 3\star(2a+3)\}$-choosable by using Lemma~\ref{lem-key}.

Let $H_1$ be a graph with vertex set $(A\cup B)$ such that $A\cap B=\emptyset$ and 
\begin{itemize}
			\item $A$ induces a ${3(a+2)}$-clique, $B$ induces a $(t-a-3)$-clique.
			\item $\{A_1, A_2,\dots,A_{a+2}\}$ is a partition of $A$ with $|A_i|=3$ for $i\in[a+2]$ and $T=\{X\subseteq A:|X\cap A_i|=2, \text{ for each } i\in[a+2]\}$. For each vertex $v \in B$, $N_A(v) \in T$, and for each  $X\in T$, \[|\{v\in B:N_A(v) =X\}|\ge \lfloor \frac {t-a-3}{|T|}\rfloor=\lfloor\frac{t-a-3}{m}\rfloor.\]			
		\end{itemize}
  
It is easy to see that such a graph $H_1$ exists.
		
		\begin{claim}
		\label{clm3}
			The graph $H_1$ is $K_t$-minor-free.
		\end{claim}
		
		\begin{proof}
			Assume that $H_1$ has a $K_t$-minor. Then there exists a collection $\mathcal{Z}$ of non-empty and pairwise disjoint subsets of $V(H_1)$ such that for each $Z\in\mathcal{Z}$, $H_1[Z]$ is connected, and for any two distinct $Z,Z' \in \mathcal{Z}$, there exists at least one edge in $H_1$ joining a vertex in $Z$ to a vertex in $Z'$. In particular, for any $Z\in\mathcal{Z}$, there are at least $(t-1)$ vertices in $V(H_1)-Z$ adjacent to vertices in $Z$.
			
			Since $|B|=t-a-3$, there are at least $(a+3)$ subsets $Z\in\mathcal{Z}$ that are contained in $A$. As the partition of $A$ has $(a+2)$ parts $A_1, A_2,\dots,A_{a+2}$ and $|A_i|=3$ for $i\in[a+2]$, there exists $Z\in\mathcal{Z}$ such that $|Z\cap A_i|\le 1$ for each $i\in[a+2]$.
			
			Let $X$ be a $2(a+2)$-subset of $A-Z$ such that $|X\cap A_i|=2$ for each  $i\in[a+2]$. Let $B_X=\{v\in B:N_A(v) =X\}$. Then 
			\begin{align*}
				|N_{H_1}(Z)|\le |V(H_1)|-|B_X|
				\le\;& 3(a+2)+(t-a-3)-\lfloor \frac {t-a-3}{m}\rfloor\\
				<\; & t+2a+4-\frac{t-a-3}{m}\\
				=\; & t+2a+4-\frac{t-a-3}{3^{a+2}}\\
				\le\; & t+2a+4-(2a+5)\\
				=\; & t-1,
			\end{align*}
		a contradiction.
	\end{proof}

Label the vertices in $A_i$ as $A_i=\{u^i_1,u^i_2,u^i_3\}$ for each $i\in[a+2]$.
			Let $$\bigcup_{i\in[2a+3]}\{d^i_1,d^i_2,d^i_3\}, \ \{b_i:i\in [t-5a-9]\}, \ \{c_i:i\in [2a+3]\},\ \bigcup_{i\in[2a+3]}\{c^i_1,c^i_2,c^i_3\}$$   be pairwise disjoint colour sets.

		Let $\psi: A \to \bigcup_{i\in[2a+3]}\{d^i_1,d^i_2,d^i_3\}$ be an injective mapping such that  for each $i\in[a+2]$ there exists $i_0\in[2a+3]$, $\psi(u^i_{j})=d^{i_0}_j$ for $j\in[3]$.
Let 
$$I(\psi)=\{i_0\in [2a+3]:\text{ there exists } i\in [a+2] \text{ such that } \psi(u^i_{j})=d^{i_0}_j \text{ for }j\in[3]\}.$$
Note that $|I(\psi)|=a+2$.
Let $L_\psi$ be the list assignment of $H_1$ defined as follows:
			\begin{itemize}
	\item[(LA')] $L_\psi(v)=\bigcup_{j\in[2a+3]}\{d^j_1,d^j_2,d^j_3\}\cup\{b_i:i\in[t-5a-9]\}$  for   $v\in A$;
	\item[(LB')] $L_{\psi}(v)=\psi(N_{A}(v))\cup\{b_i:i\in [t-5a-9]\}\cup\{c_i:i\in I(\psi)\}\cup\bigcup_{i\in[2a+3]\backslash I(\psi)}\{c^i_1,c^i_2,c^i_3\}$, for  $v\in B$.
			\end{itemize}
Let
		$$\mathcal{C}=(C_1,C_2, \ldots, C_{t-3a-6})$$
    	where $C_i=\{b_i\}$,  for $i=1,2,\dots, t-5a-9$, 
		$C_{t-5a-9+j}=\{d^j_1,d^j_2,d^j_3,c_j,c^j_1,c^j_2,c^j_3\}$, for $j=1,2,\dots,2a+3$.
		Let $\lambda = \{1 \star (t-4a-9), 3\star(2a+3)\}$.	Then $L_\psi$   is a $(\lambda, \mathcal{C})$-list assignment of $H_1$: 
	For each vertex $v$ of $B$, if $i=1,2,\ldots, t-5a-9$, $|L_\psi(v) \cap C_i|=1$; 
If $j\in I(\psi)$, $|L_\psi(v)\cap C_{t-5a-9+j}|=|\bigcup_{u^i_j\in N_{A}(v)}\psi({u}^i_{j})\cup \{c_j\}|=3$;
If $j\in [2a+3]\backslash I(\psi)$, $|L_\psi(v)\cap C_{t-5a-9+j}|=|\{c_1^j,c_2^j,c_3^j\}|=3$.
Moreover, $\psi$ is an $L_\psi$-colouring of $A$. 
   
	   \begin{claim}
		 \label{clm4}
			  $\psi$ cannot be extended to an $L_{\psi}$-colouring of $H_1$
		\end{claim}
		\begin{proof}
			Assume that $H_1$ has an $L_\psi$-colouring $\phi$ which is an extension of $\psi$.
			Then $\phi(v)=\psi(v)$, for $v\in A$, and hence 
            $\phi(v)\in\{b_i:i\in [t-5a-9]\}\cup\{c_i:i\in I(\psi)\}\cup\bigcup_{i\in[2a+3]\backslash I(\psi)}\{c^i_1,c^i_2,c^i_3\}$ for every vertex $v\in B$. 
            Thus the $(t-a-3)$-clique induced by $B$ are coloured by $(t-5a-9)+(a+2)+3(a+1)=t-a-4$ colours, a contradiction.    
		\end{proof}

Let $H_2$ be a $(t-1)$-clique and $L'$ be the $\{1\star (t-5a-9),3\star(2a+3)\}$-list assignment defined as 
$$L'(v)=\bigcup_{j\in[2a+3]}\{d^j_1,d^j_2,d^j_3\}\cup\{b_i:i\in[t-5a-9]\},$$
for each vertex $v$ of $H_2$. 

Assume $\psi$ is a proper $L'$-colouring  of $H_2$. At least 
$5a+8$ vertices of $H_2$ are coloured by colours from $\bigcup_{j \in [2a+3]}\{d_1^j,d_2^j,d_3^j
\}$. Hence there is a $3(a+2)$-clique $K_\psi=\bigcup_{i\in[a+2]}\{{u}^i_{\psi,1},{u}^i_{\psi,2},{u}^i_{\psi,3}\}$ in $H_2$ such that for each $i\in[a+2]$ there exists $i_0\in[2a+3]$, $\psi(u^i_{\psi,j})=d^{i_0}_j$ and for $j\in[3]$.

By Claim~\ref{clm4}, $L_{\psi}$ is a $\psi|_{K_\psi}$-obstacle for $(H_1,H_1[A],\mathcal{C})$.

By Lemma~\ref{lem-key}, there is a $K_t$-minor-free graph $G$ that is not $\{1\star (t-5a-9),3\star(2a+3)\}$-choosable.
\end{proof}

Next, we prove Theorems~\ref{thm:kq} and \ref{thm-1} by using a construction similar to that used by Steiner~\cite{Ste}, who proved the following lemma using a probabilistic approach.

\begin{lemma}\label{lem-2}
For every $\varepsilon\in (0,1)$, there is $n_0 = n_0(\varepsilon)$ such that for every $n \ge n_0$, there exists a graph $H$ whose vertex set $V(H)$ can be partitioned into two disjoint sets $A$ and $B$ of size $n$, and such that the following properties hold:
\begin{itemize}
	\item[1.] Both $A$ and $B$ are cliques of $H$;
	\item[2.] Every vertex in $H$ has at most $\varepsilon n$ non-neighbors in $H$;
	\item[3.] For $t = \lceil (1 + 2\varepsilon)n \rceil$, $H$ does not contain $K_t$ as a minor.
\end{itemize}
\end{lemma}

\thmkq*
		
\begin{proof}
		Let $\varepsilon\in (0,1)$ and $q\in\mathbb{N}$ be given. Pick some $\varepsilon'\in (0,1)$ such that $\frac{2-q\varepsilon'}{1+2\varepsilon'}\ge 2- \frac \varepsilon 2 $.
		Let $n_0 = n_0(\varepsilon') \in \mathbb{N}$ be as in Lemma~\ref{lem-2}, and define $t_0:=\max\{\lceil(1+2\varepsilon')n_0\rceil,\lceil \frac 6 {\varepsilon}\rceil \}.$ 
		Let $t \ge t_0$ be any given integer. Define $n:=\lfloor\frac{t}{1+2\varepsilon'}\rfloor\ge n_0$ and then $t\ge (1 + 2\varepsilon')n$.
		
		Applying Lemma~\ref{lem-2}, there exists a graph $H$ whose vertex set is partitioned into two sets $A$ and $B$ of size $n$, such that both $A$ and $B$ form cliques in $H$, every vertex in $H$ has at most $\varepsilon' n$ non-neighbors, and $H$ is $K_t$-minor-free.
		
		Let $X_1,X_2,\dots,X_q,Y_1,Y_2,\dots,Y_q$ be pairwise disjoint subsets of $\mathbb{N}$, with $|X_j|=k_j,|Y_j|=\varepsilon'n$ for each $1\le j\le q$. For each injection $c$ from vertices in $A$ to $X_1\cup X_2\cup \cdots\cup X_q$,
		let $H_c$  be a copy of $H$ with the vertex set $A_c\cup B_c$ and $G$ be a graph obtained from all copies of $H$ by identifying the different copies of $v\in A$ into a single vertex for each vertex $v\in A$. Denote the vertex set of $G$ by $A\cup \bigcup_c B_c$. Since $H$ is $K_t$-minor-free and the set $A$ forms a clique of size $n$, $G$ is $K_t$-minor-free by repeated application of Lemma~\ref{lem-3}.
		
		Consider an assignment $L:V(G)\rightarrow 2^\mathbb N$ as follows: For every vertex $x\in A$, we define $L(x):=\bigcup_{j=1}^q X_j$, and for every vertex $y\in B_c$ for some injection $c$ from vertices in $A$ to $X_1\cup X_2\cup \cdots\cup X_q$, define 
		$$L(y):=\bigcup_{j=1}^q (X_j\cup Y_j)\backslash\bigcup_{x\in A, xy\notin E(G)} c(x).$$

  Let 
  $C_1,C_2,\dots, C_q$ where $C_j=X_j \cup Y_j$ for $1\le j\le q$, and $\mathcal{C}=\{C_1,C_2, \dots, C_q\}$. Let $\lambda = \{k_1,k_2,\cdots,k_q\}$. 
  Now we show that $L$ is a $(\lambda, \mathcal{C})$-list assignment of $G$.   
		For each $1\le j\le q$, $|L(v)\cap C_i|= k_i$ if $v\in A$, and $|L(v)\cap C_i|\ge k_i-\varepsilon'n+\varepsilon'n=k_i$ if $v\in V(G)\backslash A$. 
  
  It remains to prove  that $G$ is not $L$-colourable.
			Assume to the contrary that there exist an $L$-colouring $\phi$ of $G$. Let $c$ be the restriction of $\phi$ to $A$. Then $c$ is an injection $c$ from vertices in $A$ to $X_1\cup X_2\cup \cdots\cup X_q$.
		Consider the colouring  restricted to $H_c$. Note that  $|\bigcup_{v\in V(H_c)}L(v)|=|\bigcup_{j=1}^q X_j\cup Y_j|=\sum_{j=1}^q k_j+q\cdot\varepsilon'n $. 
		Since $\frac{2-q\varepsilon'}{1+2\varepsilon'}\ge 2- \frac \varepsilon 2 $ and $n=\lfloor\frac{t}{1+2\varepsilon'}\rfloor\ge \frac{t}{1+2\varepsilon'}-1$,
		\[
		\sum_{j=1}^q k_j\le (2-\varepsilon)t
		= (2-\frac \varepsilon 2)t-\frac \varepsilon 2 t
		\le \frac{2-q\varepsilon'}{1+2\varepsilon'}t-\frac \varepsilon 2 t
		\le (2-q\varepsilon')(n+1)-\frac \varepsilon 2 t. 
		\]
		Since $t\ge t_0\ge \frac 6 \varepsilon$, 
		$$|\bigcup_{v\in V(H_c)}L(v)|=\sum_{j=1}^q k_j+q\varepsilon'n\le  (2-q\varepsilon')(n+1)-\frac \varepsilon 2 t +q\varepsilon'n< 2(n+1)-\frac \varepsilon 2 t\le 2n-1.$$
		Since $|V(H_c)|=2n$ and $|\bigcup_{v\in V(H_c)}L(v)|\le 2n-1$, there are two vertices $x,y$ in $H_c$ for which $\phi(x)=\phi(y)$. Since $A$ and $B_c$ form cliques in $H_c$, we may assume that $x \in A$ and $y \in B_c$. Now $\phi(x)=\phi(y)$ implies that  $xy\notin E(G)$. But then $\phi(x)=c(x) \notin L(y)$, a contradiction. 
		\end{proof}
		
\thmmult*
		
\begin{proof}
Let $\varepsilon\in (0,1)$ be given. Pick some $\varepsilon'\in (0,1)$ such that $\frac{2-\varepsilon'}{1+2\varepsilon'}\ge 2-\frac \varepsilon 2$.
		Let $n_0 = n_0(\varepsilon') \in \mathbb{N}$ be as in Lemma~\ref{lem-2}, and define $t_0:=\max\{\lceil(1+2\varepsilon')n_0\rceil, \lfloor \frac 6\varepsilon\rfloor\}.$ 
		Now, let $t \ge t_0$ be any given integer. Define $n:=\lfloor\frac{t}{1+2\varepsilon'}\rfloor\ge n_0$ and then $t \ge (1 + 2\varepsilon')n$. 
		
		By Lemma~\ref{lem-2}, there exists a graph $H$ whose vertex set is partitioned into two non-empty sets $A$ and $B$ of size $n$, such that both $A$ and $B$ form cliques in $H$, every vertex in $H$ has at most $\varepsilon' n$ non-neighbors, and $H$ is $K_t$-minor-free.
		
		For any fixed positive integer $m$, let $D$ be the family of all $m$-subsets of $[2nm-1]=\{1,2,\dots,2nm-1\}$. 
  For each injection  $c$ from vertices in $A$ to $D$, let $H_c$  be a copy of $H$ with the vertex set $A_c\cup B_c$ and $G$ be a graph obtained from all copies of $H$ by identifying the different copies of $v\in A$ into a single vertex for each vertex $v\in A$. Denote the vertex set of $G$ by $A\cup \bigcup_{c} B_c$. Since $H$ is $K_t$-minor-free and the set $A$ forms a clique of size $n$, $G$ is $K_t$-minor-free by repeated application of Lemma~\ref{lem-3}.
		
Consider an assignment $L:V(G)\rightarrow 2^\mathbb N$ to vertices in $G$ as follows: For every vertex $x\in A$, we define $L(x):=[2nm-1]$, and for every vertex $y\in B_c$ for some injection $c$ from vertices in $A$ to $D$, define 
$$L(y):=[2nm-1]\backslash\bigcup_{x\in A, xy\notin E(G)} c(x).$$ Recall that every vertex in $H$ has at most $\varepsilon' n$ non-neighbors. So $|L(v)|\ge 2nm-1-\varepsilon' nm$ for every vertex $v\in V(G)$.

It remains to prove that $G$ does not admit a $m$-fold $L$-colouring, which will then prove that $G$ is not $((2-\varepsilon')nm-1,m)$-choosable. Assume to the contrary that there exists an $m$-fold $L$-colouring $\phi$ of $G$. Let $c$ be the restriction of $\phi$ to $A$, and then $c$ is an injection from vertices in $A$ to $D$.
		
		Consider the colouring  restricted to the subgraph induced by $H_c:=A\cup B_c$ in $G$. Since $|V(H_c)|=2n$ and $\bigcup_{v\in V(H_c)}L(v)=[2nm-1] $, there are two vertices $x,y$ in $H_c$ which have $\phi(x)\cap \phi(y)\neq \emptyset$. Since both of $A$ and $B_c$ form a clique in $H_c$, there exists $x\in A$, $y\in B_c$ and a colour $i\in[2nm-1]$ such that $xy\notin E(G)$ and $i\in \phi(x)\cap \phi(y)$. Thus $i \in c(x)$ and hence $i \notin L(y)$, a contradiction. 
		
		Since $t\ge t_0\ge \frac 6\varepsilon$,
		\begin{align*}
			(2-\varepsilon')nm-1
			&=(2-\varepsilon') \left\lfloor\tfrac{t}{1+2\varepsilon'}\right\rfloor m-1\\
			&>(2-\varepsilon')\left(\tfrac{t}{1+2\varepsilon'}-1\right)m-1\\
			&\ge (2-\frac \varepsilon 2)tm-(2m-\varepsilon'm+1)\\
			&\ge (2-\varepsilon)tm. 
		\end{align*}
		Hence, we conclude that $G$ is a $K_t$-minor-free graph that is not $((2-\varepsilon )tm,m)$-choosable.
		\end{proof}

\subsection*{Acknowledgements} This research was partially completed at the \href{https://www.matrix-inst.org.au/events/structural-graph-theory-downunder/}{Structural Graph Theory Downunder} workshop at the Mathematical Research Institute MATRIX (November 2019).  Xuding Zhu is partially supported by NSFC Grants 11971438, U20A2068, and ZJNSFC  grant LD19A010001.

\end{document}